\documentclass[a4paper, final, notitlepage, 10pt]{article}
\usepackage{latexsym,bm}
\usepackage{amsthm}
\usepackage{amsfonts}
\usepackage{amsmath}
\usepackage[all]{xy}
\usepackage[final]{graphics}
\usepackage{fancyhdr}
\usepackage{verbatim}
\usepackage{geometry}
\usepackage{xcolor}
\usepackage{enumerate}
\usepackage{showkeys}
\usepackage{amssymb}
\usepackage{adjustbox}
\usepackage[flushleft]{threeparttable}

\newtheorem{thm}{Theorem}[section]

\newtheorem{lem}[thm]{Lemma}
\newtheorem{prop}[thm]{Proposition}
\theoremstyle{definition}

\newtheorem{rem}[thm]{Remark}

\newtheorem{ques*}[thm]{Question}
\newtheorem*{theorem*}{Question}

\numberwithin{equation}{section}

\newcommand{\thmref}[1]{Theorem~\ref{#1}}

\newcommand{\propref}[1]{Proposition~\ref{#1}}

\newcommand{\ZZ}{\mathbb{Z}}

\newcommand{\QQ}{\mathbb{Q}}

\newcommand{\mN}{\mathcal{N}}

\newcommand{\FF}{\mathbb{F}}
\renewcommand{\O}{\mathcal{O}}

\newcommand{\PP}{\mathbb{P}}

\newcommand{\mL}{\mathcal{L}}
\newcommand{\tpi}{\tilde{\pi}}
\newcommand{\Pic}{\mathrm{Pic}}

\newcommand{\NS}{\mathrm{NS}}
\newcommand{\Tors}{\mathrm{Tors}}

\newcommand{\disc}{\mathrm{disc}}
\newcommand{\Ker}{\mathrm{Ker}}

\newcommand{\Num}{\mathrm{Num}}
\newcommand{\num}{\mathrm{num}}

\title{Numerical Godeaux Surfaces with many disjoint $(-2)$-curves and Applications}
\author{Yifan Chen and YongJoo Shin}
\begin{document}
\date{}
\renewcommand{\thefootnote}{\fnsymbol{footnote}}

\maketitle

\footnotetext{Yifan Chen,
School of Mathematical Sciences, Beihang University, 9 Nansan Street,
Shahe Higher Education Park, Changping, Beijing, 102206, P. R. China
Email:~chenyifan1984@buaa.edu.cn}

\footnotetext{YongJoo Shin,
Email:~haushin@hanmail.net}

\footnotetext{{\itshape Mathematics Subject Classification (2020)}~:~14J29.}
\footnotetext{Keywords: Godeaux surface, $(-2)$-curve, involution, surface of general type.}

\begin{abstract}
In this paper, over the field of complex numbers, we prove that a numerical Godeaux surface contains at most six pairwise disjoint $(-2)$-curves, and that this bound is sharp. As an application, we refine the classification of involutions on smooth minimal surfaces of general type with $p_g=0$ and $K^2=7$: the divisorial fixed part $R$ satisfies $R^2=-1$, the involution acts trivially on $H^*(S,\mathbb{Q})$, and, if the minimal resolution of the quotient is of general type, it is a numerical Campedelli surface containing five pairwise disjoint $(-2)$-curves. Another application concerns  commuting involutions on smooth minimal surfaces of general type with $p_g=0$ and $K^2=8$.
\end{abstract}

\section{Introduction}
We work over the field of complex numbers through this paper. Numerical Godeaux surfaces, namely smooth minimal surfaces of general
type with $p_g=0$ and $K^2=1$, constitute the smallest $K^2$ case in
the geography of surfaces of general type with vanishing geometric
genus; see, in particular, Reid's foundational work
\cite{GodeauxReid}.
A smooth rational curve with self-intersection $-2$ will be called a $(-2)$-curve. The study of surfaces with many
such curves has been based on the Bogomolov--Miyaoka--Yau inequality
and its extensions to singular surfaces
\cite{Megyesi99,MQS},
as well as on binary codes associated with even sets of nodes
\cite{maxnodesBeauville,manynodes,keum,enriquesnodes}. For further study related to nodal surfaces we mention the recently published book \cite{Catanese26}.

Let $Y$ be a smooth minimal surface of general type with
$p_g(Y)=0$ and $K_Y^2=1$. Define
\[
\mu(Y):=
\max\left\{
k\ \middle|\
\text{there exist } k \text{ pairwise disjoint }(-2)
\text{-curves on }Y
\right\}.
\]
For such a surface, $c_2(Y)=11$, and Miyaoka's inequality (\cite[Page~161~(6)]{MQS}) gives
\[
3c_2(Y)-K_Y^2\geq \frac{9}{2}\mu(Y),
\]
hence $\mu(Y)\leq 7$ (cf.~\cite[Proposition~4.1]{manynodes} and \cite[Theorem~1.3]{keum}).
In his description of surfaces containing $h^{1,1}-2$ disjoint $(-2)$-
curves, Keum left the case $K^2=1$ with seven $(-2)$-curves as one of
the cases for which no example was known
\cite[Remark~4.10~(2)]{keum}.
Our main result excludes this extremal case.

\begin{thm}\label{thm:godeauxnodes}
Let $Y$ be a smooth minimal surface of general type with
$p_g(Y)=0$ and $K_Y^2=1$. Then
\[
\mu(Y)\leq 6.
\]
\end{thm}

\begin{rem}
The bound in Theorem~1.1 is sharp. In \cite{RGR26}, Rito, Gleissner and Ruhland
constructed singular $\mathbb Z/2\mathbb Z$-Godeaux surfaces with singular set $6A_1$, $4A_1+1A_3$ or $2A_1+2A_3$.
The minimal models of these surfaces satisfy $\mu(Y)=6$.
\end{rem}

The proof combines the binary code attached to the seven $(-2)$-curves
with the identification
\[
[K_Y]^\perp\simeq E_8(-1).
\]
The code provides three compatible $2$-divisibility relations and
hence the covering data for a bidouble cover. Double cover arguments
and a final discriminant computation in the unimodular numerical
lattice then yield a contradiction.

We next apply Theorem~1.1 to involutions on surfaces with
$p_g=0$ and $K^2=7$. The first examples of such surfaces were
constructed by Inoue
\cite{inoue},
and their moduli and further families have been studied in
\cite{inouemfd,chennew,calabristagnaro,chenshin21,rito15}.
Their bicanonical maps were investigated by Mendes Lopes and Pardini
\cite{bicanonical1,bicanonical2},
while Lee and the second named author classified the possible birational models of the
quotients and the corresponding branch divisors induced by an
involution
\cite{leeshin}. More recently, the first, second named authors and Zhang \cite{CSZ26} refined the classification of
Lee and the second named author. Their results, together with the earlier classification,
give the seven cases listed in
\cite[Table~1]{CSZ26}.

Theorem~1.1 rules out cases (2), (3)(a), and
(3)(b) in that list. Indeed, we obtain the following.

\begin{thm}\label{thm:K2=7}
Let $S$ be a smooth minimal projective surface with
$p_g(S)=0$ and $K_S^2=7$. Assume that $\sigma$ is an involution on
$S$. Let $R$ be the divisorial part of the fixed locus of
$\sigma$.
\begin{enumerate}[\upshape (a)]
\item $R^2=-1$.

\item $\sigma$ acts trivially on $H^*(S,\mathbb Q)$.

\item If the minimal resolution $W$ of the quotient surface
$S/\langle\sigma\rangle$ is of general type, then $W$ is a numerical
Campedelli surface containing five disjoint $(-2)$-curves.
\end{enumerate}
\end{thm}

\begin{rem}
In the notation of
\cite[Table~1]{CSZ26}, \thmref{thm:K2=7}
excludes cases (2), (3)(a), and (3)(b). Consequently,
case (1) is the only possibility when $W$ is of general type.

The numerical Campedelli alternative in Theorem~1.3(c) do exists. See \cite{chenshin21, rito15}.
\end{rem}
Finally, in connection with the use of nodal quotient surfaces in the
study of involutions on surfaces with $p_g=0$ and $K^2=8$
(see \cite{manynodes}, \cite[Remark~4.10~(2)]{keum} and \cite{dzambic_roulleau}),
we prove the following bound for elementary abelian $2$-groups.

\begin{thm}\label{thm:fakequadric}
Let $S$ be a smooth minimal surface of general type with
$p_g(S)=0$ and $K_S^2=8$. Assume that $G$ is a subgroup of
$\operatorname{Aut}(S)$,
\[
G\cong (\mathbb Z/2\mathbb Z)^r,
\]
and that each nontrivial element of $G$ has no divisorial fixed part.
Then
\[
r\leq 2.
\]
\end{thm}

The paper is organized as follows. Section~2 fixes the notation and
recalls the required lattice-theoretic facts. Section~3 proves
Theorem~1.1 by means of double covers, binary codes, and a bidouble
cover. Sections~4 and~5 prove Theorems~1.3 and~1.5, respectively.

\section{Preliminaries and notation}\label{sec:preliminaries}

Throughout this paper, all varieties are defined over the field of
complex numbers. Unless otherwise stated, a surface is assumed to be
smooth, projective, and connected. For a surface $X$, we use the
standard notation
\[
p_g(X):=h^0(X,K_X),\qquad
q(X):=h^1(X,\mathcal O_X),\qquad
\chi(\mathcal O_X):=1-q(X)+p_g(X),
\]
and denote by $\rho(X)$ the Picard number of $X$.
 For two
divisors $D_1,D_2$, we write
\[
D_1\equiv D_2.
\]
for linear equivalence.

For a surface $X$ with $p_g(X)=q(X)=0$
 the exponential exact sequence shows that
$$c_1 \colon \Pic(X)\rightarrow H^2(X,\ZZ)$$
is an isomorphism and it follows that $\NS(X)=H^2(X,\ZZ)$.
We identify these groups and denote by $\Tors(X)$ the torsion subgroup of $\Pic(X)$. Then $\Num(X)=\Pic(X)/\Tors(X)$.
For a divisor $D$ (resp. invertible sheaf $\mL$), denote by $[D]$ (resp. $[\mL]$) its class in $\Num(X)$.
 We make the following convention:
Let $D_1, \ldots, D_m$ be divisors of $W$. Denote by $(D_1, \ldots, D_m)$ the matrix of intersection numbers
$(D_i\cdot D_j)_{1\le i,j \le m}$.

By Noether's formula and the Hodge decomposition, we have $$\rho(X)=10-K_X^2.$$
Moreover, since
$$\Num(X)\cong H^2(X,\mathbb Z)/\operatorname{Tors} H^2(X,\mathbb Z),$$
Poincar\'e duality implies that the intersection pairing on $\Num(X)$ is unimodular. In particular, $$|\det (\Num(X))|=1.$$

Now let $M\subset \Num(X)$ be a sublattice of finite index. Then $$|\det (M)|=[\Num(X):M]^2\,|\det (\Num(X))|=[\Num(X):M]^2.$$
Note that $\det (M)$ is, up to sign, the square of an integer.

\section{Proof of Theorem~\ref{thm:godeauxnodes}}

This section is devoted to prove \thmref{thm:godeauxnodes}.
Assume that by contradiction that $Y$ contains seven disjoint $(-2)$-curves:
 $$C_1, \dots, C_7.$$

\subsection{Even sets consist of four $(-2)$-curves}

\begin{prop}\label{prop:KL}Assume that $\mL \in \Pic(Y)$ and $2\mL$ is linear equivalent to the sum of
four disjoint $(-2)$-curves. Then
$$H^0(Y, K_Y+\mL)=0.$$
\end{prop}
\begin{proof}
Without loss of generality, we assume that
$$C_1+C_2+C_3+C_4\equiv 2\mL.$$
There is a double cover $\tpi \colon \tilde{Z} \rightarrow Y$ branched along $C_1+C_2+C_3+C_4$ associated with the covering data above (see \cite{involution, Pardini91}).
Let $\epsilon \colon \tilde{Z}\rightarrow Z$ be the contraction of the four $(-1)$-curves $E_i:=\tpi^{-1}(C_i)$ for $i=1,2,3,4$. Then we have a commutative diagram:
\begin{displaymath}
\xymatrix{
  \tilde{Z}  \ar[r]^{\epsilon}  \ar[d]_{\tpi} & Z  \ar[d]^{\pi}\\
  Y \ar[r]^{\eta}                                 & \Sigma }
\end{displaymath}
where $\eta$ is a contraction of four $(-2)$-curves $C_j$ for $j=1,2,3,4$, and $\pi$ is a double cover $\pi \colon Z \rightarrow \Sigma$ branched along the four nodes of $\Sigma$.

We have $2K_{\tilde{Z}}\equiv \tpi^*(2K_Y+C_1+\cdots + C_4)$ and $K_Z=\pi^*K_\Sigma$.
It follows that we have $K_{\tilde{Z}}^2=-2$, $K_Z^2=2$ and $K_Z$ is nef.
 By standard double cover formulae,
\begin{align*}
    \chi(\O_Z)&=\chi(\O_{\tilde{Z}})=2\chi(\O_Y)+\frac{(\mL^2+K_Y\cdot \mL)}{2}=1,\\
    p_g(Z)&=p_g(\tilde{Z})=p_g(Y)+h^0(Y,K_Y+\mL)=h^0(Y,K_Y+\mL).
\end{align*}
Thus $Z$ is a smooth minimal surface of general type with $K^2_Z=2$ and $\chi(\O_Z)=1$. Then $p_g(Z)=q(Z)$.

If $q(Z)\ge 1$, then Debarre's inequality (i.e. $K_Z^2\ge 2p_g(Z)$ in \cite[Th\'eor\`eme 6.1]{Debarre82}) implies $$p_g(Z)=q(Z)=1.$$

The Albanese pencil $a\colon Z\rightarrow A$ of $Z$ is a fibration of curves of genus $2$ over an elliptic curve $A$ (see \cite[Theorem 5.3]{Horikawa81}, \cite{catanese_ciliberto_conf} and \cite[Theorem 1.1]{CMLP14}). Pull the Albanese pencil $a$ back to $\tilde{Z}$ as $\tilde{a}\colon \tilde{Z}\rightarrow A$. Then $E_1, \dots, E_4$ and $\tpi^*C_j$ ($j=5,6,7$) are contained in the fibres of $\tilde{a}$.

Let $\tilde{\iota}$ be the involution on $\tilde{Z}$ associated with the double covering $\tilde{\pi}$. Then $\tilde{\iota}$ induces an involution $\iota_A$ on $A$ such that
$$\tilde{a}\circ \tilde{\iota}=\iota_A\circ \tilde{a}.$$
Therefore the fibration $\tilde{a}$ descends to $Y$ as a fibration $a'\colon Y\rightarrow A/\langle \iota_A \rangle$.
Since $q(Y)=0$, $A/\langle \iota_A \rangle\cong \PP^1$. It follows that $a'$ is a fibration of curves of genus $2$ such that $C_1,\dots,C_7$ are contained in fibres of $a'$.

Denote by $F$ the general fibre of $a'$ on $Y$. We have
$$K_Y\cdot{}F=2, F^2=0, F\cdot{}C_k=0~\text{for}~k=1,2,\dots,7.$$
Then we obtain
$\det(K_Y, F, C_1, \ldots, C_7)=2^9$, which is not a square integer.
This gives a contradiction to the fact that $\Num(Y)$ is unimodular.

Therefore $q(Z)=0$, and so $p_g(Z)=0$. Thus $H^0(Y, K_Y+\mL)=p_g(Z)=0$.

\end{proof}

\subsection{Binary codes}
Recall from Section~2 that $\Num(Y)$ is an indefinite odd unimodular lattice of signature $(1,8)$.
Set $\Lambda:=[K_Y]^{\perp}$.
\begin{lem}\label{lem:numY} $\Num(Y)=[K_Y]\oplus \Lambda$, and $\Lambda$ is isomorphic to $E_8(-1)$.
\end{lem}
\begin{proof}
 Since $K_Y^2=1$, $\ZZ[K_Y]$ is primitive and so
$\Lambda$ is unimodular.
And $\Lambda$ is negative definite (resp. even) by the Hodge index theorem (resp. the adjunction formula).
Also $\mathrm{rank}(\Lambda)=\rho(Y)-1=8$.
Therefore $\Lambda\cong E_8(-1)$ by \cite[Chapter V, 2.3 The definite case, Example (i)]{serre}.
\end{proof}

 Since $\Num(Y)$ is an indefinite odd unimodular lattice of signature $(1,8)$, the intersection form on $\Num(Y)$ induces a bilinear form on the $\FF_2$-vector space $\Num(Y)/2\Num(Y)\cong \FF_2^9$, which can be represented by the identity matrix.
So this bilinear form is non-degenerate. It follows that any isotropic subspace of $\Num(Y)/2\Num(Y)$ has dimension $\le \left \lfloor 9/2 \right \rfloor =4$.

As in \cite[Section~2]{enriquesnodes}, we consider the kernel $$V:=\Ker(\FF_2^7\rightarrow \Pic(Y)/2\Pic(Y))$$ of the homomorphism defined by
$$(x_1,\dots,x_7) \mapsto \O_Y(x_1C_1+\cdots+x_7C_7)\mod 2\Pic(Y).$$
Denote by $\Gamma'$ the sublattice of $\Num(Y)$ generated by $[C_1], \dots, [C_7]$ and by
$\Gamma$ the primitive closure of $\Gamma'$ in $\Num(Y)$.
We also deal with the kernel
 $$V_{\num}:=\Ker(\FF_2^7\cong \Gamma'/2\Gamma' \rightarrow \Num(Y)/2\Num(Y))$$
 of the homomorphism defined by
$$(x_1,\dots,x_7) \mapsto x_1[C_1]+\cdots+x_7[C_7] \mod 2\Num(Y).$$
Note that the image of $\Gamma'/2\Gamma'$ in $\Num(Y)/2\Num(Y)$ is isotropic, therefore
$$\dim_{\FF_2}V_\num \ge 7-4=3.$$

By \cite[Page~675, (2.1)]{enriquesnodes}, we have
$$2^7=2^{2\dim_{\FF_2} V_{\num}} \disc(\Gamma).$$
Here $\disc(\Gamma)$ is the discriminant of $\Gamma$.

It follows that
\begin{equation}\label{dimVnum and disc(Gamma)}
\dim_{\FF_2}V_\num=3 \textrm{ and }\disc(\Gamma)=2.
\end{equation}

Since $\Num(Y)$ is a free abelian group, from the exact sequence
$$0 \rightarrow \Tors(Y) \rightarrow \Pic(Y) \rightarrow \Num(Y)\rightarrow 0,$$
we obtain the following exact sequence
$$0 \rightarrow \Tors(Y)/2\Tors(Y) \rightarrow \Pic(Y)/2\Pic(Y) \rightarrow \Num(Y)/2\Num(Y) \rightarrow 0.$$

Since $\Tors(Y)$ is a cyclic group of order $\le 5$,
$\dim_{\FF_2} \Tors(Y)/2\Tors(Y)=0$ or $1$.
The commutative diagram
\begin{align*}
\xymatrix{
V \ar"1,2" \ar"2,1" & V_{num} \ar"2,2"\\
\Pic(Y)/2\Pic(Y) \ar"2,2"& \Num(Y)/2\Num(Y)
}
\end{align*}
and the exact sequence above show that
\begin{equation}\label{dimV>1}
\dim_{\FF_2}V\ge \dim_{\FF_2} V_\num-1= 2.
\end{equation}

Moreover, for any nonzero $x=(x_1,\dots,x_7)\in V$ $$\frac{1}{2}\sum_{i\in I_x}[C_i]\in [K_Y]^{\perp}\cong E_8(-1) \textrm{ by Lemma \ref{lem:numY}},$$ where $I_x:=\{i\in\{1,2,\dots,7\}\ |\ x_i\neq0\}$. Since $E_8(-1)$ is even, $$-\frac{|I_x|}{2}=\left(\frac{1}{2}\sum_{i\in I_x}[C_i]\right)^2\in 2\mathbb{Z}.$$ Thus we obtain \begin{equation}\label{Ix}|I_x|=4.\end{equation}

\begin{prop}\label{prop:coverdata}After possibly renumbering $C_1, \ldots, C_7$, there exists $\mL_1, \mL_2, \mL_3$ such that
$$\mL_1+C_1+C_2=\mL_2+\mL_3,~~\mL_2+C_3+C_4=\mL_1+\mL_3,~~\mL_3+C_5+C_6=\mL_1+\mL_2.$$
\end{prop}
\begin{proof}
By (\ref{dimV>1}) we take a subspace $V'$ of dimension $2$ of $V$. Then, by (\ref{Ix}), after possibly renumbering $C_1, \ldots, C_7$,
there exists $\mL_1, \mL_2, \mL_3$ such that
\begin{align}
2\mL_1&\equiv C_3+C_4+C_5+C_6 \label{eq:L1}\\
2\mL_2&\equiv C_1+C_2+C_5+C_6 \label{eq:L2}\\
2\mL_3&\equiv C_1+C_2+C_3+C_4 \label{eq:L3}.
\end{align}
It follows that
$$2\mL_1+2C_1+2C_2\equiv 2(\mL_2+\mL_3).$$
Therefore
$$\eta:=\mL_1+C_1+C_2-(\mL_2+\mL_3) \in \Pic(Y)[2].$$
Replace $\mL_1$ by $\mL_1+\eta$, then \eqref{eq:L1}-\eqref{eq:L3} still hold and
\begin{align}
\mL_1+C_1+C_2\equiv \mL_2+\mL_3, \label{eq:L4}
\end{align}
Add \eqref{eq:L4} to \eqref{eq:L2} (resp~\eqref{eq:L3}), we obtain

$$\mL_1+\mL_2\equiv \mL_3+C_5+C_6~~(\text{resp.}~ \mL_1+\mL_3\equiv \mL_2+C_3+C_4).$$
\end{proof}

\begin{prop}\label{prop:Gammaperp}There exists $\mN \in \Pic(Y)$ such that $\mN^2=-2, K_Y\cdot{}\mN=0$,
and inside $\Lambda$,
$$\Gamma^{\perp}=\langle [\mN]\rangle.$$
\end{prop}
\begin{proof}
We have seen $\disc(\Gamma)=2$ in (\ref{dimVnum and disc(Gamma)}). Since $\Lambda$ is unimodular and primitive, inside $\Lambda$, $\disc(\Gamma^{\perp})=2$ by \cite[Chapter~2~ Lemma~2.12]{mordellweil}.
Since $\Gamma^\perp$ is even, negative definite and of rank $1$,
it follows that any generator $\textbf{n}$ of $\Gamma^{\perp}$ satisfies
$\textbf{n}^2=-2$.
Take $\mN \in \Pic(Y)$ such that $[\mN]=\textbf{n}$.
\end{proof}

\subsection{The bidouble cover and a contradiction}\label{section:bidouble trick}

Let $\tpi \colon \tilde{Z} \rightarrow Y$ be the bidouble cover according to the data in \propref{prop:coverdata} (see \cite[Section 1]{FC84}, \cite[Theorem 2]{FC99} and  \cite{Pardini91}). Then $\tpi^{-1}(C_j)$ is a disjoint union of two $(-1)$-curves for $j=1, 2, \dots, 6$. Let $\epsilon \colon \tilde{Z} \rightarrow Z$ be the blowdown of these twelve $(-1)$-curves.  We have a commutative diagram:
\begin{displaymath}
\xymatrix{
  \tilde{Z}  \ar[r]^{\epsilon}  \ar[d]_{\tpi} & Z  \ar[d]^{\pi}\\
  Y \ar[r]^{\eta}                                 & \Sigma }
\end{displaymath}
where $\eta$ is a contraction of six $(-2)$-curves $C_j$ for $j=1,2,\dots,6$, and $\pi$ is a bidouble cover $\pi \colon Z \rightarrow \Sigma$ branched along the six nodes of $\Sigma$.

We have $2K_{\tilde{Z}}\equiv \tpi^*(2K_Y+C_1+\cdots + C_6)$ and $K_Z=\pi^*K_\Sigma$. It follows that  $K_{\tilde{Z}}^2=-8$, $K_Z^2=4$ and $K_Z$ is nef. Since
\[
    p_g(Z)=p_g(\tilde{Z})=p_g(Y)+\sum_{i=1}^3h^0(Y, K_Y+\mL_i)=\sum_{i=1}^3h^0(Y, K_Y+\mL_i)
\]
we obtain $p_g(Z)=0$ by \propref{prop:KL}. Therefore $Z$ is a smooth minimal surface of general type with $K_Z^2=4$, $p_g(Z)=q(Z)=0$. From Section~2, we have $\rho(Z)=6$.

Note that $\tpi^{-1}(C_7)$ is a disjoint union of four $(-2)$-curves $\tilde{N}_1, \tilde{N}_2, \tilde{N}_3, \tilde{N}_4$. By \propref{prop:Gammaperp}, $\tpi^*\mN$ is orthogonal to $K_{\tilde{Z}}$, $\tilde{\pi}^*(C_j)$ ($j=1,\dots,6$), $\tilde{N}_1, \tilde{N}_2, \tilde{N}_3, \tilde{N}_4$. It follows that there is $\mathcal{N}' \in \Pic(Z)$ such $\tilde{\pi}^*\mathcal{N}=\epsilon^*\mathcal{N}'$. Note that $K_Z\cdot{}\mathcal{N}'=0$ and $\mathcal{N}'^2=-8$.

Denote by  $N_i=\epsilon_\ast \tilde{N}_i$ ($i=1,2,3,4$).
Then $N_1, \cdots, N_4$ are disjoint $(-2)$-curve and $N_i\cdot{}\mathcal{N}'=0$ ($i=1,2,3,4$).
Therefore
$$\det(K_Z, N_1, N_2, N_3, N_4, \mN')=4\cdot{}(-2)^4\cdot{}(-8)=-2^9,$$
whose absolute value is not a square integer.
This gives a contradiction to the fact that $\Num(Z)$ is unimodular
and complete the proof of \thmref{thm:godeauxnodes}.

\section{Proof of Theorem~\ref{thm:K2=7}}

We have
\[
R^2=\pm 1
\qquad\text{and}\qquad
\operatorname{tr}\bigl(\sigma^*\mid H^2(S,\QQ)\bigr)=2-R^2
\]
by \cite[Lemma~2.6]{commuting} and \cite[Lemma~4.2]{manynodes}, respectively.

Let
\[
\pi\colon S\longrightarrow \Sigma:=S/\langle\sigma\rangle
\]
be the quotient map, and let
\[
\eta \colon W\longrightarrow \Sigma
\]
be the minimal resolution of the singularities of $\Sigma$. Let $B$ be the divisorial part of the branch locus of $\pi$, and let $B_0$ denote its pull-back to $W$. Since the singularities of $\Sigma$ are precisely the images of the isolated fixed points of $\sigma$, whereas $B$ is the image of the divisorial fixed locus $R$, the divisor $B$ is disjoint from the singularities of $\Sigma$. Hence $B_0$ is simply the strict transform of $B$ under $\eta$, and therefore $B_0^2=B^2$.
Moreover, since $\pi$ is a double cover branched along $B$, we have $\pi^*B=2R$. Thus
\[
B_0^2=B^2=2R^2.
\]

Note that $R^2=1$ (i.e. $B_0^2=2$) can occur only in case~(2) of \cite[Table~1]{CSZ26}. By Theorem~1.1, this case is excluded. Therefore
\[
R^2=-1
\qquad\text{and}\qquad
\operatorname{tr}\bigl(\sigma^*\mid H^2(S,\QQ)\bigr)=3.
\]
Recall from Section~2 that $\dim H^2(S,\QQ)=\rho(S)=3$. This proves (a) and (b).

We next observe that, in cases~3(a) and~3(b) of \cite[Table~1]{CSZ26}, the exceptional curve on $W$ is disjoint from the seven $(-2)$-curves $C_j$ associated with the seven isolated fixed points of the involution $\sigma$ on $S$. Indeed, let $W'$ be the minimal model of $W$. Since $K_{W}^2=0$ and $K_{W'}^2=1$
the contraction $t\colon W\longrightarrow W'$
contracts the unique $(-1)$-curve $E$ on $W$. For each $j$, we have
\[
0=K_W\cdot C_j
  =(t^*K_{W'}+E)\cdot C_j
  =t^*K_{W'}\cdot C_j+E\cdot C_j.
\]
Since $K_{W'}$ is nef, we have $t^*K_{W'}\cdot C_j\ge 0$, and hence $E\cdot C_j\le 0$.

On the other hand, $E$ and $C_j$ are distinct irreducible curves, so $E\cdot C_j\ge 0$.
Therefore
\[
E\cdot C_j=0
\]
for every $j$, and thus $E$ is disjoint from all the curves $C_j$.

The images
\[
t(C_1),\ldots,t(C_7)
\]
are seven disjoint $(-2)$-curves on $W'$. Hence $W'$ is a numerical Godeaux surface containing seven disjoint $(-2)$-curves, contradicting Theorem~\ref{thm:godeauxnodes}. Thus cases~3(a) and~3(b) are excluded.

Moreover, if $W$ is of general type, the preceding exclusions show that only case~(1) of \cite[Table~1]{CSZ26} can occur. Therefore $W$ is a smooth minimal surface of general type, a numerical Campedelli surface, with $p_g(W)=0$ and $K_W^2=2$ containing five disjoint $(-2)$-curves.

\section{Proof of Theorem~\ref{thm:fakequadric}}
Assume by contradiction that $r\ge 3$. We may take a subgroup $H \le G$ such that $H\cong (\ZZ/2\ZZ)^3$. As pointed out in \cite[Theorem~4.10~(2)]{keum},
the minimal resolution of the quotient surface $S/H$ is a numerical Godeaux surface with seven disjoint $(-2)$-curves.
This gives a contradiction to \thmref{thm:godeauxnodes}.

\section*{Acknowledgments}
The first named author would like to thank Yong Hu and JongHae Keum for many helpful discussions.
The bidouble covering trick (Subsection~\ref{section:bidouble trick}) in the proof of Theorem~\ref{thm:godeauxnodes} is suggested by ChatGPT-pro~5.6. The authors used Google Gemini for minor English language editing of specific sentences.

\end{document}